\documentclass[12pt,reqno, a4paper]{article}
\usepackage{}
\usepackage{bbm}
\usepackage[colorlinks]{hyperref}
\usepackage{amsfonts}
\usepackage{mathrsfs}
\usepackage{amsmath}
\usepackage{amsthm}
\usepackage{amssymb}
\usepackage{amsmath}
\usepackage{bbm}
\usepackage{amsfonts}
\usepackage{mathrsfs}
\IfFileExists{ajr.sty}{\usepackage{ajr}}{}
\usepackage{url}
\usepackage{amssymb, latexsym, amsthm}

\newcommand{\rat}[2]{{\textstyle\frac{#1}{#2}}}

\newcommand{\Z}[1]{\ensuremath{e^{#1 t}\star}}
\newcommand{\Za}[1]{\ifcase#1 undef0
  \or undef1
  \or\Z{-\rat{27}{10}}
  \or\Z{-\rat{38}{5}}
  \else undefinf \fi}
\newcommand{\Zb}[1]{\ifcase#1 undef0
  \or\Z{-\rat{2}{\epsilon}}
  \or\Z{-\rat{5}{\epsilon}}
  \or\Z{-\rat{10}{\epsilon}}
  \else undefinf \fi}

\usepackage{amsthm}
\newtheorem{theorem}{Theorem}
\newtheorem{lemma}[theorem]{Lemma}
\newtheorem{corollary}[theorem]{Corollary}
\newtheorem{definition}[theorem]{Definition}
\newtheorem{example}[theorem]{Example}
\theoremstyle{remark}
\newtheorem{remark}{Remark}
\topsep=\parskip

\newcommand{\Lip}{\operatorname{Lip}}

\begin{document}
\title{Center manifolds for infinite dimensional random dynamical systems}


\author{Xiaopeng Chen
\thanks{Beijing International Center for Mathematical Research, Peking University, Beijing, 100871, \textsc{China}. \protect\url{mailto:
chenxiao002214336@yahoo.cn} }
\and
Anthony J. Roberts\thanks{School of Mathematical Sciences, University of Adelaide, Adelaide, 5005, \textsc{Australia}.
\protect\url{mailto:anthony.roberts@adelaide.edu.au}}
\and
Jinqiao Duan\thanks{Institute for Pure and Applied Mathematics (IPAM), University of California, Los Angeles,  CA 90095, USA \& Department of Applied Mathematics, Illinois Institute of Technology,   Chicago, IL 60616, USA.
\protect\url{mailto:jduan@ipam.ucla.edu}}
}

\date{\today}

\maketitle

\begin{abstract}
Stochastic center manifolds theory are crucial in modelling  the dynamical behavior of  complex systems under stochastic influences. 
A multiplicative ergodic theorem on Hilbert space is proved to be satisfied to the  exponential trichotomy
condition.
Then the existence of stochastic center manifolds  for infinite dimensional random dynamical systems is shown under the assumption of  exponential trichotomy. The theory provides a support for the discretisations of nonlinear stochastic partial differential equations with space-time  white noise.
\end{abstract}

\section{Introduction}
The theory of  centre manifolds plays an important role  in  the deterministic dynamical systems and has been proved a tremendous
applications such as bifurcation~\cite{Har,Mar}. It  has been developed by many people(e.g. Carr~\cite{Carr},   Kelly~\cite{Kel},  Vanderbauwhede~\cite{Van}).
It is  important to study the stochastic  case of center manifolds since  in many applications  the  dynamical  systems are  influenced by  noise. Arnold~\cite{Arn98}  summarised various invariant manifolds  on  finite dimensional random dynamical systems. Mohammed and Scheutzow~\cite{MohScheu99} focused
 on the existence of local  stable and unstable  manifolds for stochastic differential equations driven by semimartingales. Boxler~\cite{Box}  proved the existence of  stochastic
  center manifold  for  finite dimensional random dynamical systems by using the multiplicative ergodic theorem and discrete random map. Roberts~\cite{Rob,Rob1,Rob2}  assumed  existence of  stochastic center manifolds  for infinite stochastic partial differential equations in exploring  the  interactions of microscale noises and their macroscale modelling. The natural problem is  to show the nature of stochastic center  manifolds  on infinite dimensional spaces.

  In the present paper, we study the stochastic center manifolds for infinite dimensional random dynamical systems.  We prove the existence of stochastic  center  manifolds  on infinite dimensional random dynamical systems. We extend the results of  the stochastic center manifold theory on finite dimensional random dynamical systems~\cite{Arn98,Box} to the infinite dimensional random dynamical systems.  However,  Boxer's  proof is not suitable in infinite dimension case since the Ascoli's theorem~(\cite{Box}, Lemma 4.4) cannot directly extend to the general infinite dimensional space. Recently we explain the existence and properties of stochastic center manifolds  for a class of stochastic evolution equations~\cite{Chen} by  the existence of exponential trichotomy. However it is not suitable for stochastic partial differential equations  driven by nonlinear multiplicative niose.

In recent years, there are some results about the invariant manifolds for  random dynamical systems that generated from  stochastic partial differential equations.
 Duan and others~\cite{Car2,Duan1,Duan2}  presented stable and unstable invariant manifolds for a class of stochastic partial differential
 equations driven by one dimensional Brownian motion under the assumption of exponential dichotomy or pseudo exponential dichotomy. Stochastic inertial manifolds which  generalized from center-unstable manifolds on   finite dimensional  spaces  are constructed  by different methods~\cite{Ben, Car1, Prato}. Chen et al.\cite{Chen} proved the existence and its properties of  center manifolds for a class of  stochastic evolution equations with linearly multiplicative noise. It is a natural way to consider the random dynamical systems generated by nonlinear multiplicative
stochastic partial differential equations.  However, the problem  is still open since one cannot apply Kolmogorov¡¯s theorem  to ensure the stochastic flow
property~(\cite{Prato1}, pp. 246-248). For the infinite dimensional random dynamical systems, a classic result is the multiplicative ergodic theorem(\textsc{met}). Ruelle~\cite{Rue2} proved the stable and unstable manifolds  in Hilbert space for differentiable dynamical systems  by the technique of multiplicative ergodic theorem. Mohammed~\cite{Moh} gave  a details of the extension of Ruelle~\cite{Rue2} to prove the stable and unstable manifold of a class of semilinear stochastic evolution  equations and semilinear stochastic partial differential equations by showing the existence of perfect cocycles.  For the  discussion of  stable and unstable manifolds in discrete random dynamical systems we refer to Li and Lu~\cite{Li}, Lian and Lu~\cite{Lia}.

The concept of exponential trichotomy is important for center manifold theory  in infinite dimensional dynamical systems and non-autonomous systems~\cite{Bar, Gal, Pli}.
The existence of exponential trichotomy means  that the space is split  into three subspaces: center subspace, unstable subspace and stable subspace.
Center manifold theory is based on the assumption of exponential trichotomy~\cite{Pli}.
We first introduce the exponential trichotomy for random dynamical systems. Then we change the random dynamical systems under the context of equaions which we  consider.    Our basic tool is the multiplicative ergodic theorem. We introduce the  multiplicative ergodic theorem(\textsc{met}) on two side time in an infinite dimensional Hilbert space. Then the condition of exponential trichotomy holds if the Lyapunov exponents satisfies some gap condition. Various versions of  multiplicative ergodic theorem for  random dynamical systems  on finite dimension space  have been summarized by Arnlod~\cite{Arn98}.  A discrete version on multiplicative ergodic theorem  for random dynamical systems  on Hilbert space have recently been proved by Ruelle~\cite{Rue2}. Mohammed et al.~\cite{Moh} gave  a details of the extension to one side continuous time from  a discrete version of multiplicative ergodic theorem~\cite{Rue2}.

Constructing accurate and efficient models of nonlinear stochastic  partial differential equations is an important task. There  have been some theory developments in seeking numerical approximation of nonlinear stochastic partial differential equations, for example, finite difference~\cite{Gy,Yoo}, Galerkin approximation~\cite{Blo,Gy1}. Roberts applies the stochastic center manifold theory to derive more accurately and efficiently discretisations of nonlinear stochastic  partial differential equations~\cite{Rob3,Rob1,Rob2}.   However, the assumption of existence  stochastic center manifolds in infinite dimensional spaces is needed. In the present paper, we give a theory to support spatial discretisations of  nonlinear stochastic  partial differential equations by stochastic  center manifolds. We analysis the nonlinear stochastic Burgers equation driven by space-time white noise as an example. We show that stochastic nonlinear Burgers equation driven by space-time white noise generate a random dynamical systems and there exists a stochastic center manifold.

Section \ref{pre}  show how to adapt the results to two side time discrete and  continuous random dynamical systems. By the multiplicative ergodic theorem we have proved, we split the Hilbert space into three  subspaces: center subspace, unstable subspace and stable subspace. Then we introduce the exponential trichotomy and give a condition of existence of exponential trichotomy in Section \ref{Exp}. Used the Lyapunov--Perron method, Section~\ref{Sec3}  proves the existence of  stochastic center manifolds both for  discrete and continuous random dynamical systems  under the assumption of exponential trichotomy. This result applied to stochastic partial differential equations driven by  by one dimensional Brownian motion and the random dynamical systems generated by nonlinear stochastic Burgers equation driven by space-time white noise  in Section \ref{Sec4}.

\section{Preliminaries}\label{pre}
We recall below the definition of a cocycle  in Hilbert space. Let $(\Omega,\mathcal{F},\mathbb{P})$  be a complete probability space. Suppose
$\theta :\mathbb{R}\times \Omega \to \Omega$ is a group of $\mathbb{P}$-preserving ergodic
 transformations on  $(\Omega,\mathcal{F},\mathbb{P})$.

Let $H$ be a real separable Hilbert space with  norm $|\cdot|$ and
Borel $\sigma$-algebra $\mathcal {B}(H)$.

Let $\mathbb{T}=\mathbb{Z}$ or $\mathbb{R}$. A random dynamical system $(U, \theta)$  on $H$
 is a $\mathcal {B} (\mathbb{R}) \otimes \mathcal{F} \otimes \mathcal {B}(H) $-
measurable random mapping
$U:\mathbb{T}  \times \Omega \times H \to H$
with the following properties:
\begin{enumerate}
\item   $U(t_1+t_2,\omega)=
U(t_2,\theta (t_1,\omega))\circ U(t_1,\omega)$
for all $t_1,t_2 \in \mathbb{T} $, all $\omega \in \Omega$.
\item  $U(0,\omega)x=x$ for all $x \in H, \omega \in \Omega$.
\end{enumerate}

A
random variable
\begin{equation*}\label{eqA5}
X:(\Omega,\mathcal{F})\to
(\mathbb{R}^+\setminus \{0\},\mathcal{B}(\mathbb{R}^+\setminus\{0\}))
\end{equation*}
  is called {\em tempered from above} if

\begin{equation*}
    \limsup_{t\to\pm\infty}\frac{\log^+ X(\theta_t\omega)}{|t|}=0
\end{equation*}
for $\omega$ contained in a $\{\theta_t\}_{t\in\mathbb{R}}$
invariant set of full measure($t\to-\infty $ applies only to two-sided time). Such a random variable $X$ is
called {\em tempered from below}  if $X^{-1}$ is tempered from above. $X$ is call \emph{tempered}  if both tempered from above and tempered from below.
Arnold~\cite{Arn98} proved that the random variable is tempered if and if it is $\varepsilon$-slowly varying for some $\varepsilon\geq 0$.

  Boxler~\cite{Box} used the random norm in order to obtain the random variables $K_1(\omega)$,  $K_2(\omega)$, $K_3(\omega)$ are constants. However, for  comparing convenience in the same metric, we used the Hilbert norm and prove the random variables are slowly varying.

We  introduce an  infinite dimensional version of the
multiplicative ergodic theorem(\textsc{met}) with two side continuous time. 
\begin{theorem}\label{t2}
Let $U$ be a linear random dynamical system of compact operators
on $H$ satisfying the following integrability condition
\begin{eqnarray}\label{inte}
   \mathbb{E}\sup_{0\le t\le 1}\log^+\|U^{\pm}(t,\omega)\|+
    \mathbb{E}\sup_{0\le t\le
    1}\log^+\|U^{\pm}(1-t,\theta_t\omega)\|<\infty.
\end{eqnarray}
Then there is a measurable set  $\Omega_0 \in \mathcal{F}$  such
that  $\theta_t(\Omega_0)\subset \Omega_0$ for all $t \in \mathbb{R}$, and for
 each $\omega \in \Omega_0$, the limit
$$ \Lambda (\omega):= \lim_{t \to\pm \infty} [U(t,\omega)^* \circ U(t,\omega)]^{1/(2t)}  $$
exists in the uniform operator norm.  Each linear operator $\Lambda (\omega)$ is compact, non-negative and self-adjoint with a discrete spectrum
\begin{eqnarray*}
exp(\lambda_1)>exp(\lambda_2)>exp(\lambda_3)>\ldots,
\end{eqnarray*}
where the $\lambda_i$'s are distinct and non-random. Then there exist linear spaces $H= W_1(\omega)\oplus\cdots \oplus W_\infty(\omega)$,  $\dim W_i(\omega)=d_i$, $i=1,2,\ldots$, such that
\begin{align*}
 & \lim_{t\to\pm \infty}\frac{1}{t}\log\|U(t,\omega)x\|=\left\{ \begin{array}{ll}\lambda_i
   & \textrm{if  $x\in W_{i}(\omega),$}\\-\infty & \textrm{if $x\in W_{\infty}(\omega)$,}
\end{array} \right.
    \\&
    U(t,\omega)W_{i}(\omega)\subset W_{i}(\theta_t\omega).
\end{align*}
for $t\in  \mathbb{R}$ and all $\omega\in \Omega_0$.
\end{theorem}
\begin{remark}
The numbers $\lambda_1,\,\lambda_2,\cdots$ are called the {\em
Lyapunov exponents} associated to random dynamical system $U$. The set of these numbers
forms the
{\em Lyapunov spectrum}.
\end{remark}
\begin{proof}
Mohammed~\cite{Moh0,Moh} gave  a details of the extension to one side continuous time from  a discrete version of multiplicative ergodic theorem~\cite{Rue2}.  We show how to adapt to the Mohammed~\cite{Moh0,Moh}'s results to two side continuous time. The technique is from the discrete  result~\cite{Arn98,Rue1, Rue2}.

From the integration condition \eqref{inte} and  Theorem 2.1.1 of Mohammed~\cite{Moh}, we know that the  Lyapunov exponents associated to random dynamical system $U^{-}(t,\omega)=U^{-1}(t,\theta_t^{-1}\omega)$ is $-\lambda_1<-\lambda_2<-\lambda_3<\cdots$, and the corresponding eigenspace $F_{-i}(\omega)$ with $d_{i}=\dim F_{-i}(\omega)$, $\omega\in \Omega_1$. The  Lyapunov exponents associated to random dynamical system $U^{+}(t,\omega)=U(t,\omega)$ is $\lambda_1>\lambda_2>\lambda_3>\cdots$, and the corresponding eigenspace $F_{i}(\omega)$ with $d_i=\dim F_{i}(\omega)$, $\omega\in \Omega_2$.  Define
   \begin{eqnarray*}
V_r(\omega)=[\oplus _{j=1}^{r-1}F_j(\omega)]^{\perp},\\
V_{-r}(\omega)=\oplus_{j=1}^{r} F_{-j}(\omega),
\end{eqnarray*}
for $r=1,2,\ldots$.
Then
\begin{eqnarray*}
V_{-1}(\omega)\subset V_{-2}(\omega)\subset V_{-3}(\omega)\subset\cdots \subset V_{-\infty}(\omega)=H
\end{eqnarray*}
for  $\Omega_1\in \mathcal{F}$  such
that  for all $t \in \mathbb{R}^+$, $\theta_t^{-1}(\Omega_1)\subset \Omega_1$,  codim $ V_{r+1}(\omega)=\dim V_{-r}(\omega)$, and
\begin{align*}
 & \lim_{t\to \infty}\frac{1}{t}\log\|U^{-1}(t,\theta_t^{-1}\omega)x\|=\left\{ \begin{array}{ll}-\lambda_i
   & \textrm{if  $x\in V_{-i-1}(\omega)\setminus
     V_{-i}(\omega),$}\\\infty & \textrm{if $x\in V_{-1}(\omega)$,}
\end{array} \right.
    \\&
    U^{-1}(t,\theta^{-1}_t\omega)V_{-i}(\omega)\subset V_{-i}(\theta^{-1}_t\omega).
\end{align*}
That is   $\theta_t (\Omega_1)\subset \Omega_1$, and
\begin{align*}
 & \lim_{t\to -\infty}\frac{1}{t}\log\|U(t,\omega)x\|=\left\{ \begin{array}{ll}\lambda_i
   & \textrm{if  $x\in V_{-i-1}(\omega)\setminus
     V_{-i}(\omega),$}\\-\infty & \textrm{if $x\in V_{-1}(\omega)$,}
\end{array} \right.
    \\&
    U(t,\omega)V_{-i}(\omega)\subset V_{-i}(\theta_t\omega),
\end{align*}for all $t \in \mathbb{R}^-$.

We show that
for almost $\omega\in \Omega_1 \cap \Omega_2=:\Omega_0$,
 $V_{r+1}(\omega)\cap V_{-r}(\omega)=\emptyset$,  $V_{r+1}(\omega)\oplus V_{-r}(\omega)=H$.

Let $B:=\{\omega\in \Omega_0: V_{r+1}(\omega)\cap V_{-r}(\omega)\neq \emptyset\}$. Select $\omega\in B$ such that $v\in V_{r+1}(\omega)\cap V_{-r}(\omega)$. Given $\delta>0$, let $B_n$ be the subset of $B$ such that if $\omega\in B_n$,  for all $v\in V_{r+1}(\omega)\cap V_{-r}(\omega)$,
\begin{eqnarray}
\|U(n,\omega)v\|\leq \|v\|\exp[n(\lambda_{r+1}+\delta)],\label{th1eq1}\\
\|U^{-1}(n,\omega)v\|\leq \|v\|\exp[n(-\lambda_{r}+\delta)],
\end{eqnarray}
If  $v\in V_{r+1}(\omega)\cap V_{-r}(\omega)$, then $U(n,\omega)v\in V_{r+1}(\theta_n\omega)\cap V_{-r}(\theta_n\omega)$. If $\omega\in \theta_{n}^{-1}B_n\cap B_n$, we obtain
\begin{eqnarray}
\|v\|=\|U^{-1}(n,\omega)U(n,\omega)v\|\leq \|U(n,\omega)v\|\exp(n(-\lambda_{r}+\delta)).\label{th1eq2}
\end{eqnarray}
Equations \eqref{th1eq1} and \eqref{th1eq2} implies $\lambda_r-\lambda_{r+1}\leq 2\delta$. Since $\mathbb{P}(B_n\cap \theta_n^{-1}B_n)\rightarrow \mathbb{P}(B)$ and $\delta$ is arbitrary,  we conclude $\mathbb{P}(B)=0$.     $V_{r+1}(\omega)\oplus V_{-r}(\omega)=H$ follows codim $ V_{r+1}=\dim V_{-r}$.

Let $W_i(\omega)=V_{i}(\omega)\cap V_{-i}(\omega)$, we obtain $H=V_1(\omega)\cap(V_{-1}(\omega)\oplus V_2(\omega))\cap \cdots \cap V_{-\infty}(\omega)=W_1(\omega)\oplus\cdots W_{\infty}(\omega)$.
\end{proof}

\begin{remark}
  If the time is discrete, Ruelle ~\cite{Rue2} prove a similar conclusion replacing of the assumption \eqref{inte} to $\mathbb{E}\log^+\|U^{\pm}(n,\omega)\|<\infty.$
\end{remark}


\section{MET implies exponential trichotomy} \label{Exp}

In this section,
we  consider the exponential trichotomy for  random dynamical systems $U(t,\omega)$, $t\in \mathbb{T}$  in a Hilbert space.

First we introduce the exponential trichotomy,  which generalizes the exponential dichotomy~\cite{Cop, Duan1, Duan2, Sac}.

\begin{definition}
$U(t, \omega)$, $t\in \mathbb{T}$  is said to be exponential trichotomy  if there exists a $\theta_t$-invariant set
$\tilde \Omega \subset \Omega$ of full measure such that for each
$\omega\in \tilde \Omega$, the phase space $H$ splits into
\[
H=E^s(\omega)\oplus E^c(\omega)\oplus E^u(\omega)
\]
of closed subspaces satisfying
\begin{enumerate}
\item This splitting is invariant under $\;U(t, \omega)$:
\begin{align*}
&U(t, \omega)  E^s(\theta_t\omega)\subset E^s(\theta_t\omega), \\
&U(t, \omega)  E^c(\theta_t\omega)\subset E^c(\theta_t\omega), \\
&U(t, \omega) E^u(\theta_t\omega)\subset E^u(\theta_t\omega);\\
\end{align*}
 \item  There are
 constants $\alpha >\gamma>
\beta>0$, and tempered   variables $K^s(\omega):\tilde
\Omega \to (0, \infty)$, $K^c(\omega):\tilde
\Omega \to (0, \infty)$ and  $K^u(\omega):\tilde
\Omega \to (0, \infty)$  such that
\begin{align}
\label{a}\|U(t, \omega)\|& \leq K^s(\omega)
\exp{[-\alpha t]}
\quad\hbox{for } t\geq 0, \\
\label{b}\|U(t, \omega)\|& \leq K^c(\omega)
\exp{\gamma|t|}
\quad\hbox{for } t\in  \mathbb{R},\\
\label{c}\|U(t,
\omega) \|&
\leq K^u(\omega) \exp{\beta t} \quad\hbox{for } t\leq  0.
\end{align}
\end{enumerate}
\end{definition}

\begin{lemma}\label{12}
Suppose that the following exponential integrability condition is
satisfied
\begin{eqnarray} \label{eq21}
   \mathbb{E}\log^+\sup_{0\le t\le 1}\|U^{\pm}(t,\omega)\|+
    \mathbb{E}\log^+\sup_{0\le t\le
    1}\|U^{\pm}(1-t,\theta_t\omega)\|<\infty.
\end{eqnarray}
Then there exists a $\{\theta_t\}_{t\in\mathbb{R}}$
invariant set of full $\mathbb{P}$-measure and a constant
$H(\omega) $ such that
\begin{equation*}
    \|U^u(t,\omega)\|\geq H(\omega) \exp({a
    t})\quad\text{for }t\leq 0,\,\omega\in\tilde\Omega
\end{equation*}
for a sufficiently large $a$.
\end{lemma}
\begin{proof}

Since
\begin{equation*}
    \|U^u(t,\omega)\|\leq \|U^-(t,\omega)\|,
\end{equation*}
$U^-(t,\omega)= U(-t,\omega)$, we conclude that
$\mathbb{E}\log^+\|U^u(-1,\omega)\|<\infty$.
By subadditive ergodic theorem,  there exists a set
of measure one such that
\begin{equation}\label{x1}
    \lim_{i\to-\infty}\frac{1}{i}\log
    \|U^u(i,\omega)\|=
 a.
\end{equation}

Set
\begin{equation*}
    \Omega_n^1:=\{\omega\in\Omega:\lim_{i\to-\infty}
    \frac{1}{i}\log \|U^u(i,\theta_n\omega)\|=
     a\}\in\mathcal{F},\\
\end{equation*}

\begin{equation*}
    \Omega^1:=\bigcap_{n\in\mathbb{Z}}\Omega_n^1.
\end{equation*}

The set $ \Omega^1$ is $\{\theta_t\}_{t\in\mathbb{Z}}$--invariant and has
probability one.

Let $D_1(\omega)=\log\sup_{-1\le t\le 0}\|U^u(t,\omega)\|$. Since $ \mathbb{E}D^+_1(\omega)< \infty$, by  Borel-Cantelli lemma,  there exists a full measurable set $\Omega^2$ so that
\begin{equation*}
    \limsup_{i\to-\infty}\frac{1}{i}D_1(\theta_{i+n}\omega)=0,\qquad
    n\in\mathbb{N}.
\end{equation*}

 For $\omega\in\Omega^1\cap\Omega^2$,
\begin{align}\label{eqn1}
\begin{split}
   \log\|U^u(t,\omega)\|= &
    \log\|U^u(t-[t],\theta_{[t]}\omega)\|
    +\log\|U^u([t],\omega)\|.
 \end{split}
\end{align}
 Thus
$\limsup_{t\to-\infty}\log\|U_\lambda^u(t,\omega)\|/t\ge
a$. We obtain the conclusion.
\end{proof}

Now we show that the multiplicative ergodic theorem (\textsc{met}) in Section \ref{pre} implies  exponential
trichotomy in  Hilbert spaces.
\begin{theorem}\label{t3}
Suppose that the following exponential integrability condition is
satisfied
\begin{eqnarray}\label{eq22}
   \mathbb{E}\log^+\sup_{0\le t\le 1}\|U^{\pm}(t,\omega)\|+
    \mathbb{E}\log^+\sup_{0\le t\le
    1}\|U^{\pm}(1-t,\theta_t\omega)\|<\infty.
\end{eqnarray}
and suppose that $\alpha, \beta, \gamma \in\mathbb{R}$ is not contained in the
Lyapunov spectrums such that $\cdots<\lambda_i<-\beta<-\gamma<\lambda_j<\gamma<\alpha<\cdots<\lambda_2<\lambda_1$.
 Then there exists a
$\{\theta_t\}_{t\in\mathbb{R}}$ invariant set $\tilde \Omega$ of
full measure such that for $\omega\in\tilde \Omega$ we have the
following properties: There exist linear spaces
$E^u(\omega),\,E^c(\omega),\,E^s(\omega)$ such that
\begin{equation*}
    H=E^u(\omega)\oplus E^c(\omega)\oplus E^s(\omega).
\end{equation*}

There exist   tempered
random variables $K^s(\omega)$,  $K^c(\omega)$ and  $K^u(\omega)$ such that,
\begin{align*}
    \|U^s(t,\omega) \|&\leq
    K^s(\omega) \; \exp({-\alpha t}),  \quad t\geq 0;\\
     \|U^c(t,\omega)\|&\leq
    K^c(\omega) \; \exp({\gamma |t|}), \quad t \in \mathbb{R};\\
    \|U^u(t,\omega)\|&\le
    K^u(\omega) \; \exp({\beta t}), \quad t\leq 0.
\end{align*}
\end{theorem}
\begin{proof} Let the Lyapunov spectrums  $\lambda_1>\lambda_2>\cdots>\lambda_+>\lambda_j=0>\lambda_->\lambda_{i+1}>\cdots$,  $E^s(\omega)=W_{i+1}\oplus W_{i+2}\oplus \ldots$.  We start with $K^s(\omega)$. Define
\begin{equation*}
    K^s(\omega):=\sup_{t\in
    \mathbb{R}^+}\frac{\|U^s(t,\omega)\|}{\exp[{(\lambda_-+\varepsilon) t}]},
\end{equation*}
where $-\alpha=\lambda_-+\varepsilon$.  Then by the multiplicative ergodic theorem the random variable is finite.
Let $F_n(\omega) = \log \| U^s(n, \omega)\|$.  Since \begin{equation*}\mathbb{E}F_1^+(\omega)=\mathbb{E}\log^+\|U^s(1,\omega)\|<+\infty \end{equation*}and $F_{m+n}(\omega)=\log\|U^s(m+n,\omega)\|\leq\log\|U^s(m,\theta_n\omega)\|\|U^s(n,\omega)\|= F_n(\omega)+F_m(\theta_n\omega)$.  By subadditive
ergodic theorem,  there exists a
$\{\theta_t\}_{t\in\mathbb{Z}}-$invariant measurable function
$F(\omega)$ such that

\begin{equation}\label{kingman1}
 F(\omega)=\lim_{n\to \infty} \frac1{n}F_n(\omega)
 =  \lim_{n\to \infty} \frac{1}{n} \log\|U^s(n, \omega)\| \leq
\lambda_-.
\end{equation}

By a consequence of Kingman's subadditive ergodic theorem\cite[Corollary A.2]{Rue2}, for every $\epsilon>0$, there
is a finite-valued random variable $K_{\epsilon}(\omega)$   such that when $n>m$,
\begin{equation}\label{kingman2}
\log\|U^s(n-m, \theta_m\omega)\| \leq
(n-m)\lambda_-+n{\epsilon} +K_{\epsilon}(\omega), \;\;  a.s.
\end{equation}

Let $D(\omega)=\log^+\sup_{0\le t\le 1}\|U(t,\omega)\|+\log^+\sup_{0\le t\le
    1}\|U(1-t,\theta_t\omega)\|$, then $ \mathbb{E}D(\omega)<\infty$.  From the  Borel-Cantelli lemma,
there is a measurable set $\Omega_1$ such that $\mathbb{P}(\Omega_1)=1$, and
\begin{eqnarray*}
  \lim\limits_{i\to\infty}\frac 1 i D(\theta_{i+n}\omega)=\limsup_{i\to\infty}\frac 1 i D(\theta_{i+n}\omega)=0, \quad \omega\in \Omega_2, n\in \mathbb{N}.
\end{eqnarray*}
From
 \begin{align*}
    \|U^s(t,\theta_r\omega)\|&\le
    \|U^s(1+t-[t]-1-[r]+r,\theta_{[r]+[t]}\omega)\|\\
    &\times
    \|U^s([t]-1,\theta_{1+[r]}\omega)\|\times\|U^s(1-r+[r],\theta_r\omega)\|\\
    &\leq
    \exp[{D(\theta_{[r]+[t]-1}\omega)}]\exp[{K_\varepsilon(\omega)+\epsilon([r]+1)+(\lambda_-+{\epsilon})([t]-1)}]\\&\ \quad
\exp[{D(\theta_{[r]}\omega)}].
\end{align*}
From the
 definition of $K^s(\omega)$, $$\lim\limits_{r\to \infty}\frac {\log^+ K^s(\theta_r\omega)}{r}=0.$$ So the random variable $K^s(\omega)$ is  tempered.

Next we prove $K^u(\omega)$ is tempered. Let $E^u(\omega)=W_1(\omega)\oplus W_2(\omega)\oplus \ldots \oplus W_k(\omega) \ldots$ for $\lambda_k\geq \lambda_+$.
Since
\begin{equation}
\lim_{t\to -\infty}\frac{1}{t}\log\|U^u(t,\omega)x\|\geq\lambda_+>0, \quad x\in E^u(\omega),
\end{equation}
then there exists a $\varepsilon>0$ and $t_1(x,\varepsilon,\omega)<0$ such that $\|U(t,\omega)x\|\geq e^{t(\lambda_+-\varepsilon)}$ for $t<t_1$. Let $\beta=\lambda_+-\varepsilon$.   
 From Lemma \ref{12}, assume that
\begin{equation}\label{eq9}
    \|U^u(t,\omega)\|\geq H(\omega)  \exp[{a t}]\quad
    \text{for }\quad t\leq 0,\,\omega\in\Omega,
\end{equation}
for a measurable function $H(\omega)$.
 We show that
\begin{equation*}
   {K^u(\omega)}:=\sup_{t\leq
    0}\frac{\|U^u(t,\omega)\|}{\exp[{(\lambda_+-\varepsilon)  t}]}
    \end{equation*}
is a tempered  random variable in $(0,\infty)$. 
Since
\begin{equation*}
    \exp({a t})\le
    \frac{\|U^u(t,\omega)\|}{H(\omega) } \quad\text{for
    any } t\leq 0.
\end{equation*}
We then see that for $s<0$
\begin{align*}
  { K^u(\theta_s\omega)}\exp({a s})&=\sup_{t\le
    0}\frac{\|U^u(t,\theta_s\omega)\|}{\exp[{(\lambda_+-\varepsilon)  t}]} \exp({a s})\\
    &\le \sup_{t\le
    0}\frac{\|U^u(t,\theta_s\omega)\|
    \|U^u(s,\omega)\|} {
    \exp[{(\lambda_+-\varepsilon) (t+s)}]} \frac 1 {H(\omega)} \exp[{(\lambda_+-\varepsilon) s}]\\
    &\le
    \sup_{t\le
    0}\frac{
    \|U^u(t+s,\omega)\|}{\exp[{(\lambda_+-\varepsilon)  (t+s)}]}\frac 1 {H(\omega)} \exp[{(\lambda_+-\varepsilon)  s}]\\
    &=
    \frac{K^u(\omega)} {H(\omega) } \exp[{(\lambda_+-\varepsilon)  s}]
\end{align*}
which goes to zero for $s\to-\infty$.

Similarly  there exists a tempered random variable $
K^c(\omega)$.

\end{proof}







\section{Stochastic center manifolds} \label{Sec3}

We introduce the definition of stochastic center manifolds.
A basic tool of proving the existence of  stochastic center manifolds is to define an appropriate function space, which is a Banach space.

\begin{definition}
A random set~$M(\omega)$ is called an (forward) invariant set for a random
dynamical system~$\varphi(t,  x,\omega)$ if
$\varphi(t,M(\omega),\omega)\subset M(\theta_t\omega)$ for
$t\geq 0$\,.
If we can represent~$M(\omega)$ by a graph of a  (Lipschitz)
mapping from the center subspace to its complement,
$h^c(\cdot,\omega): H^c\to H^u\oplus H^s$,
such that
$M(\omega)=\{v + h^c(v,\omega) \mid v\in H^c\}$,
$h^c(0,\omega)=0$\,, and the tangency condition that  the derivative $Dh^c(0,\omega)=0$\,, $h^c(v,\cdot)$ is measurable for every $v\in H^c$,
then $M(\omega)$~is called a (Lipschitz) center
manifold, often denoted as~$M^c(\omega)$.
\end{definition}
First we prove the eixistenc of center manifolds for discrete time random dynamical systems.
  Let $U(n,x,\omega)$ is the linearization of random dynamical system $\varphi(n,x,\omega)$, i.e. the Fr\'{e}chet derivative  $D\varphi(n,x,\omega)$ at point $x$. Then  $U(n,x,\omega)$ is also a random  dynamical system. Let $F(n,x,\omega)=\varphi(n,x,\omega)-U(n,x,\omega)$, then
 \begin{eqnarray} \label{dis1} \varphi(n,x,\omega)=U(n,x,\omega)+F(n,x,\omega).  \end{eqnarray}  
Assume $F(1,x,\omega)$ is Lipschitz continuous on $H$ with Lipschitz constant $\operatorname{Lip} F(\omega)$ and   $F(1,0,\omega) = 0$, $DF(1,0,\omega) = 0$. We introduce a modified equation by using a cut-off technique. Let
$\sigma(s)$ be a $C^\infty$ function from $(-\infty,\infty)$ to $[0,1]$ with
\begin{eqnarray*}
&\sigma(s) = 1 \,\mbox{for} \, |s|\leq 1,  \quad
\sigma(s) = 0 \, \mbox{for}  \, |s|\geq 2, \\
&\sup \limits_{s\in \mathbb{R}} |\sigma'(s)|\leq 2.
 \end{eqnarray*}
Let $\rho:\Omega\to (0,\infty)$ be a  tempered random variable such that $G(x,\omega)=\sigma(\frac {|x|}{\rho(\omega)})F(1,x,\omega)$,  We  assume it to be
Lipschitz continuous on $H$, that is,
 \begin{eqnarray*}
|G(x_1,\omega) -G(x_2,\omega)| \leq \operatorname{Lip} F(\omega)\rho(\omega)|x_1 - x_2|
 \end{eqnarray*}
with the sufficiently small Lipschitz constant $\operatorname{Lip} F(\omega)\rho(\omega) > 0$.
We show existence of a
center manifold for the discrete random dynamical system $\phi(n,x,\omega)$.

For each $\eta>0$\,,  we denote the Banach
space
\begin{align*}
D_\eta={}&\left \{ \phi:\mathbb{Z}\to H \mid  \sup_{n\in \mathbb{Z}}\exp\left[{-\eta |n|}\right] | \phi(n)|
 < \infty\right\}
\end{align*}
with the norm
\[\left|\phi\right|_{D_\eta}=\sup_{n\in \mathbb{Z}}\exp\left[{-\eta |n|}\right] |\phi(n)|.
\]
Let
\[N^c(\omega)=\left\{x_0\in H \mid \varphi(\cdot, x_0, \omega) \in D_\eta\right\},
\]
where  $\varphi(n,  x_0,\omega)$~is  the orbit  of the random dynamical system $\varphi(n, x, \omega)$ with
initial data $\varphi(0,x_0,\omega)= x_0$\,.
\begin{theorem} \label{Thm(3.2)} Suppose $U(n,\omega)$ satisfies the exponential trichotomy.
 If   $\gamma < \eta<\min\{\beta, \alpha\}$ such that the nonlinearity term is sufficiently small,
\begin{eqnarray}\label{pre2}&K(\theta\omega)  \operatorname{Lip} F(\omega)\rho(\omega) {(\frac{1}{\eta-\gamma}+\frac{1}{\beta-\eta}
+\frac{1}{\alpha-\eta})}=:\rho'(\omega)<
1\,, \label{dpr1}\end{eqnarray} where $K(\omega)=\max\{K^s(\omega),K^c(\omega),K^u(\omega)\}$, then there exists a  center manifold for
 the random  dynamical systems $\varphi(n,\omega,x)$,  which is written as the graph
\[
N^c(\omega) = \{v+h^c(v,\omega)\mid v \in H^c\},
\]
where $h^c(\cdot,  \omega) : H^c\to  H^u\oplus H^s$ is a Lipschitz continuous mapping from the center subspace and
satisfies $h^c(0,\omega)=0$, $Dh^c(0,\omega)=0$\,. 
\end{theorem}
\begin{proof}
We claim that $x_0 \in N^c(\omega)$ if and only if  there  $\varphi(\cdot, x_0, \omega)\in
D_\eta$  with
\begin{align}\label{eq(d3.2)}\begin{split}
\varphi(n, x_0,\omega)={}&U^c(n,v,\omega)+\sum\limits_{i=0}^{n-1}U^c(n-1-i,\theta^{i+1}\omega)G^c(\theta^i\omega,x_i)
 \\&{} -\sum\limits_{i=n-1}^{\infty}U^u(n-1-i,\theta^{i+1}\omega)G^u(\theta^i\omega,x_i)\\&{}+\sum\limits_{i=-\infty}^{n-1}U^s(n-1-i,\theta^{i+1}\omega)
 G^s(\theta^i\omega,x_i)\ \ \mbox{for} \, {n\neq 0},\\
\varphi(0, x_0,\omega)=&{}v -\sum\limits_{i=-1}^{\infty}U^u(-1-i,\theta^{i+1}\omega)G^u(\theta^i\omega,x_i)\\&{}+\sum\limits_{i=-\infty}^{-1}U^s(-1-i,\theta^{i+1}\omega)G^s(\theta^i\omega,x_i),
\end{split}
\end{align}
where $v = P^c x_0$.

  Consider $x_0\in N^c(\omega)$, let
  $x_n=\varphi(n,x_0,\omega)$. Using equation \eqref{dis1}
and induction, $x_n$ satisfies the discrete variation of constants formula
  \begin{eqnarray}\label{dis2}
    x_n=U(n-k,\theta^k\omega)x_k+\sum\limits_{i=k}^{n-1}U(n-1-i,\theta^{i+1}\omega)G(\theta^i\omega,x_i),
  \end{eqnarray}
  for each $k<n$, and
    \begin{eqnarray}\label{dis3}
    x_n=U(n-k,\theta^k\omega)x_k-\sum\limits_{i=n-1}^{k}U(n-1-i,\theta^{i+1}\omega)G(\theta^i\omega,x_i),
  \end{eqnarray}
  for each $k>n$.
Project $x_n$ to each subspace,
\begin{align}&\label{eq(d3.3)}
P^cx(n, x_0, \omega) = U^c(n,\omega)
v + \sum\limits_{i=0}^{n-1}U^c(n-1-i,\theta^{i+1}\omega)G^c(\theta^i\omega,x_i).
\\&
\label{eq(d3.4)}
P^u x_n =U^u(n-k,\theta^k\omega)P^ux_k-\sum\limits_{i=n-1}^{k}U^u(n-1-i,\theta^{i+1}\omega)G^u(\theta^i\omega,x_i).
\\&\label{eq(d3.44)}
P^s x_n =U^s(n-k,\omega)P^s x_k+\sum\limits_{i=k}^{n-1}U^s(n-1-i,\theta^{i+1}\omega)G^s(\theta^i\omega,x_i).
\end{align}
Since $x_n \in D_\eta$\,, we have for $n < k $ that the magnitude
\begin{align*}
 | U^u(n-k,\theta^k\omega)x_k^u|\leq K(\theta^k\omega)\exp[\beta(n-k)]\exp(\eta k)|\varphi|_{D_\eta}
  \\ \to 0 \quad\text{as }k \to +\infty.
\end{align*}
For  $n > k $\,,
\begin{align*}
| U^s(n-k,\theta^k\omega)x_k^s|\leq K(\theta^k\omega)\exp[-\alpha(n-k)]\exp(\eta k)|\varphi|_{D_\eta}
   \\
   \to 0 \quad\text{as }k \to -\infty.
\end{align*}
Then, taking the two separate limits $\tau \to \pm  \infty$   in~(\ref{eq(d3.4)}) and~\eqref{eq(d3.44)} respectively,
\begin{align}&\label{eq(d3.5)}
P^ux_n=-\sum\limits_{i=\infty}^{n-1}U^u(n-1-i,\theta^{i+1}\omega)G^u(\theta^i\omega,x_i).
\\&\label{eq(d3.55)}
P^sx_n=\sum\limits_{i=-\infty}^{n-1}U^s(n-1-i,\theta^{i+1}\omega)G^s(\theta^i\omega,x_i).
\end{align}
Combining~(\ref{eq(d3.3)}),  (\ref{eq(d3.5)}) and~(\ref{eq(d3.55)}), we have~(\ref{eq(d3.2)}).
The converse  follows from a direct computation.

For each $x_n\in D_\eta$ wich $v=P^cx_0$. We define a map $y_n=Y^c(x_n,v)$
\begin{align}\label{eq(d3.22)}\begin{split}
Y^c(x_n,v):={}&U^u(n,\omega)v+\sum\limits_{i=0}^{n-1}U^c(n-1-i,\theta^{i+1}\omega)G^c(\theta^i\omega,x_i)\\&{}
 -\sum\limits_{i=\infty}^{n-1}U^u(n-1-i,\theta^{i+1}\omega)G^u(\theta^i\omega,x_i) \\&{} \sum\limits_{i=-\infty}^{n-1}U^s(n-1-i,\theta^{i+1}\omega)G^s(\theta^i\omega,x_i).
\end{split}
\end{align}
$ J^c$~is
well-defined from $D_\eta\times H^c$ to the functions space~$D_{\eta}$.
For each
$x_n, \bar x_n \in D_{\eta}$\,, we have  that for $\gamma < {\eta}<\min\{\beta, \alpha\}$,
\begin{align}\label{eq(d3.6)}
\begin{split}
&| Y^c(x_n,v ) -  Y^c (\bar x_n, v ) |_{D_\eta} \\
&\leq\sup_{n\in \mathbb{Z}} \Bigg\{\exp\left[
-{\eta} |n| \right] \bigg |\sum\limits_{i=0}^{n-1}
U^c(n-i-1,\theta^{i+1}\omega)
(P^cG(\theta^i\omega)x_n-P^cG(\theta^i\omega)\bar x_n)
\\& \quad
-\sum\limits_{i=\infty}^{n-1}U^s(n-i-1,\theta^{i+1}\omega)(P^uG(\theta^i\omega)x_n-P^uG(\theta^i\omega)\bar x_n)\\& \quad +\sum\limits_{i=-\infty}^{n-1}U^u(n-i-1,\theta^{i+1}\omega)(P^sG(\theta^i\omega)x_n-P^sG(\theta^i\omega)\bar x_n)\bigg|
\Bigg\}\\
&\leq\sup_{n\in \mathbb{Z}} \Bigg\{  |x_n-\bar
{x_n}|_{D_\eta}\bigg|\sum\limits_{i=0}^{n-1}
K(\theta^{i+1}\omega)\operatorname{Lip} F(\theta^{i}\omega)\rho(\theta^{i}\omega) \exp[{(\gamma-{\eta})|n-i-1|}]
\\&  \quad
-\sum\limits_{i=\infty}^{n-1}K(\theta^{i+1}\omega)\operatorname{Lip} F(\theta^{i}\omega)\rho(\theta^{i}\omega)\exp[{({\beta-\eta})(n-i-1)}]\\& \quad
+\sum\limits_{i=-\infty}^{n-1}K(\theta^{i+1}\omega)\operatorname{Lip} F(\theta^{i}\omega)\rho(\theta^{i}\omega)\exp[{({\eta}-\alpha)(n-i-1)}]\bigg |
\Bigg\}\\
&\leq   |x_n-\bar
x_n|_{D_\eta}\rho'(\omega).
\end{split}
\end{align}

From equation~\eqref{eq(d3.6)}, $ Y^c$~is Lipschitz continuous on~$D_\eta$.
By the theorem's precondition~\eqref{dpr1},
$Y^c$~is a uniform contraction with respect to the parameter~$v$.
By the uniform contraction mapping principle,
for each $v \in H^c$,  the mapping $Y^c (\cdot , v )$ has a
unique fixed point $x(\cdot,v, \omega) \in D_\eta$\,.
Combining equation~\eqref{eq(d3.22)} and equation~\eqref{eq(d3.6)},
\begin{align}\label{eq(d3.7)}
&|x(\cdot ,  v,   \omega) - x (\cdot,  \bar v, \omega
)|_{D_\eta}\nonumber \\&{}\quad \leq \frac{K(\omega)}{1-\rho'(\omega) \left(\frac{1}{\eta-\gamma}+\frac{1}{\beta-\eta}
+\frac{1}{\alpha-\eta}\right)}|v - \bar v |,
\end{align}
for each
 fixed point~$x(\cdot,v, \omega)$.
 Then
$x(n,   \cdot, \omega )$~is Lipschitz from the center subspace~$H^c$ to slowly varying  functions~$D_\eta$.
 $x(\cdot,  v, \omega  )\in D_\eta$ that satisfies the equation~(\ref{eq(d3.2)}).
Since $x(\cdot , v,   \omega)$ can be an $\omega$-wise limit of the
iteration of contraction mapping~$Y^c$ starting at~$0$ and $Y^c$
maps a $\mathcal{F}$-measurable function to a $\mathcal{F}$-measurable function,
$x(\cdot, v,   \omega)$~is $\mathcal{F}$-measurable.
Combining $x(\cdot, v, \omega)$~is  continuous with respect to~$H^c$, we have $x(\cdot , v,   \omega)$ is measurable with respect to~$(\cdot, v, \omega)$.

Let $h^c (v,\omega) := P^s x (0, v, \omega)\oplus P^u x(0, v, \omega)$.
Then from equation \eqref{eq(d3.2)},
\begin{align*}
h^c(v, \omega)&=-\sum\limits_{i=-1}^{\infty}U^u(-1-i,\theta^{i+1}\omega)G^u(\theta^i\omega,x_i)\\&{}+\sum\limits_{i=-\infty}^{-1}U^s(-1-i,\theta^{i+1}\omega)G^s(\theta^i\omega,x_i)\,.
\end{align*}
We see that~$h^c$ is $\mathcal{F}$-measurable and  $h^c(0,\omega)=0$\,. Now we show $Dh^c(0,\omega)=0$
By the theorem's precondition~\eqref{dpr1},
there exists a small number $\delta>0$ such that $\gamma <
\eta-\eta'<\min\{\beta,\alpha\}$ and for all  $ 0 \leq \eta' \leq 2\delta$\,,
\begin{eqnarray*}
K(\theta\omega)  \operatorname{Lip} F(\omega)\rho(\omega) \small{\left(\frac{1}{(\eta-\eta')-\gamma}+\frac{1}{\beta-(\eta-\eta')}+\frac{1}{\alpha-(\eta-\eta')}\right)}<
1\,.
\end{eqnarray*}
Thus, $Y^c(x_n,v)$ is a uniform contraction in $D_{\eta-\eta'}
\subset  D_\eta$ for any $0\leq \eta'\leq 2\delta$\,.
Therefore,
$\varphi(\cdot, v,\omega)\in D_{\eta-\eta'}$\,.
For $v\in H^c$, we
define two operators: let
\[
Sv=U(n,\omega)v \,,
\]
and for $\phi\in {C_{\eta-\delta}}$ let
\begin{align}\label{eq(d3.221)}\begin{split}
T\phi={}&\sum\limits_{i=0}^{n-1}U^c(n-1-i,\theta^{i+1}\omega)DG^c(\theta^i\omega,\varphi(i,v_0,\omega))\phi_i
 \\&{} -\sum\limits_{i=n-1}^{\infty}U^u(n-1-i,\theta^{i+1}\omega)DG^u(\theta^i\omega,\varphi(i,v_0,\omega))\phi_i\\&{}+\sum\limits_{i=-\infty}^{n-1}U^s(n-1-i,\theta^{i+1}\omega)
 DG^s(\theta^i\omega,\varphi(i,v_0,\omega))\phi_i\ \ \mbox{for} \, {n\neq 0},\\
T\phi_0=&{} -\sum\limits_{i=-1}^{\infty}U^u(-1-i,\theta^{i+1}\omega)DG^u(\theta^i\omega,\varphi(i,v_0,\omega))\phi_i\\&{}+\sum\limits_{i=-\infty}^{-1}U^s(-1-i,\theta^{i+1}\omega)
DG^s(\theta^i\omega,\varphi(i,v_0,\omega))\phi_i,
\end{split}
\end{align}

From the assumption,
$S$~is a bounded linear operator from center subspace~$H^c$ to Banach space~$D_{\eta-\delta}$.
 $T$~is a bounded linear
operator from~${D_{\eta-\delta}}$ to itself and
\begin{eqnarray*}
\|T\| <
1\,,
\end{eqnarray*}
which implies that the operator $\operatorname{id}-T$ is invertible in~${D_{\eta-\delta}}$.
For $v, v_0\in H^c$, we set
\begin{align*}
\begin{split}
I={}&\sum\limits_{i=0}^{n-1}U^c(n-1-i,\theta^{i+1}\omega)[G^c(\theta^i\omega,\varphi(i,v,\omega))-G^c(\theta^i\omega,\varphi(i,v_0,\omega)) \\&{} -D
G^c(\theta^i\omega,\varphi(i,v_0,\omega))(\varphi(i,v,\omega)-\varphi(i,v_0,\omega))]
 \\&{} -\sum\limits_{i=n-1}^{\infty}U^u(n-1-i,\theta^{i+1}\omega)[G^u(\theta^i\omega,\varphi(i,v,\omega))-G^u(\theta^i\omega,\varphi(i,v_0,\omega)) \\&{} -D
G^u(\theta^i\omega,\varphi(i,v_0,\omega))(\varphi(i,v,\omega)-\varphi(i,v_0,\omega))]\\&{}+\sum\limits_{i=-\infty}^{n-1}U^s(n-1-i,\theta^{i+1}\omega)
 [G^s(\theta^i\omega,\varphi(i,v,\omega))-G^s(\theta^i\omega,\varphi(i,v_0,\omega)) \\&{} -D
G^s(\theta^i\omega,\varphi(i,v_0,\omega))(\varphi(i,v,\omega)-\varphi(i,v_0,\omega))]\ \ \mbox{for} \, {n\neq 0},\\
I_0=&{} -\sum\limits_{i=-1}^{\infty}U^u(-1-i,\theta^{i+1}\omega)[G^u(\theta^i\omega,\varphi(i,v,\omega))-G^u(\theta^i\omega,\varphi(i,v_0,\omega)) \\&{} -D
G^u(\theta^i\omega,\varphi(i,v_0,\omega))(\varphi(i,v,\omega)-\varphi(i,v_0,\omega))]\\&{}+\sum\limits_{i=-\infty}^{-1}U^s(-1-i,\theta^{i+1}\omega)
[G^s(\theta^i\omega,\varphi(i,v,\omega))-G^s(\theta^i\omega,\varphi(i,v_0,\omega)) \\&{} -D
G^s(\theta^i\omega,\varphi(i,v_0,\omega))(\varphi(i,v,\omega)-\varphi(i,v_0,\omega))],
\end{split}
\end{align*}

We obtain
\begin{align}\label{eq(4.1)}
\begin{split}
&\varphi(\cdot,v, \omega)-\varphi(\cdot,v_0,
\omega)-T (\varphi(\cdot,v,
\omega)-\varphi(\cdot,v_0,
\omega))
=S(v-v_0) + I,\\
\end{split}
\end{align}
which yields
\begin{align*}
\varphi(\cdot,v, \omega)-\varphi(\cdot,v_0,
\omega)=(\operatorname{id}-T)^{-1}S(v-v_0)+(\operatorname{id}-T)^{-1}I.
\end{align*}
If  $|I|_{{{D_{\eta-\delta}}}}=o(|v-v_0|)$ as $v\to
v_0$, then  $\varphi(\cdot,v, \omega) $ is differentiable in~$v$ and its
derivative satisfies $D_v \varphi(n, v,\omega)\in L(H^c,
{{D_{\eta-\delta}}})$,  where $L(H^c, {{D_{\eta-\delta}}})$ is the
usual space of bounded linear operators and
\begin{align}\label{eq(d4.2)}
\begin{split}
D_v \varphi(n, v,\omega)={}&\sum\limits_{i=0}^{n-1}U^c(n-1-i,\theta^{i+1}\omega)DG^c(\theta^i\omega,\varphi(i,v,\omega))D_v \varphi(i, v,\omega)
 \\&{} -\sum\limits_{i=n-1}^{\infty}U^u(n-1-i,\theta^{i+1}\omega)DG^u(\theta^i\omega,\varphi(i,v,\omega))D_v \varphi(i, v,\omega)\\&{}+\sum\limits_{i=-\infty}^{n-1}U^s(n-1-i,\theta^{i+1}\omega)
DG^s(\theta^i\omega,\varphi(i,v,\omega))D_v \varphi(i, v,\omega)\ \ \mbox{for} \, {n\neq 0},\\
D_v \varphi(o, v,\omega)=&{} -\sum\limits_{i=-1}^{\infty}U^u(-1-i,\theta^{i+1}\omega )DG^c(\theta^i\omega,\varphi(i,v,\omega)))D_v \varphi(n, v,\omega)\\&{}+\sum\limits_{i=-\infty}^{-1}U^s(-1-i,\theta^{i+1}\omega)
DG^s(\theta^i\omega,\varphi(i,v,\omega))D_v \varphi(i, v,\omega),
\end{split}
\end{align}
Now we prove that \begin{eqnarray}|I|_{{{D_{\eta-\delta}}}} = o(|v-v_0|), \quad |I_0|_{{{D_{\eta-\delta}}}} = o(|v-v_0|)\label{Claim1}\end{eqnarray} as
$v\to v_0$\,.
We divide~$I$ into several sufficient small parts.
Let $N$ be a large  positive number to be chosen
later and define the following ten integrals
\begin{align*}
I_1={}&\exp[{(-(\eta-\delta))|n|}]\sum\limits_{i=N}^{n-1}U^c(n-1-i,\theta^{i+1}\omega)[G^c(\theta^i\omega,\varphi(i,v,\omega))\\&{}-G^c(\theta^i\omega,\varphi(i,v_0,\omega)) -D
G^c(\theta^i\omega,\varphi(i,v_0,\omega))(\varphi(i,v,\omega)-\varphi(i,v_0,\omega))],\\
\end{align*}
for $n> N$\,.
\begin{align*}
I_1'={}&\exp[{(-(\eta-\delta))|n|}]\sum\limits_{i=n-1}^{-N}U^c(n-1-i,\theta^{i+1}\omega)[G^c(\theta^i\omega,\varphi(i,v,\omega))\\&{}-G^c(\theta^i\omega,\varphi(i,v_0,\omega)) -D
G^c(\theta^i\omega,\varphi(i,v_0,\omega))(\varphi(i,v,\omega)-\varphi(i,v_0,\omega))],
\end{align*}
for $n< -N$\,.
\begin{align*}
I_2={}&{}\exp[{(-(\eta-\delta))|n|}]\sum\limits_{i=0}^{N}U^c(n-1-i,\theta^{i+1}\omega)[G^c(\theta^i\omega,\varphi(i,v,\omega))\\&{}-G^c(\theta^i\omega,\varphi(i,v_0,\omega)) -D
G^c(\theta^i\omega,\varphi(i,v_0,\omega))(\varphi(i,v,\omega)-\varphi(i,v_0,\omega))],
\end{align*}
for $0<n \leq N$\,.
\begin{align*}
I_2'={}&\exp[{(-(\eta-\delta))|n|}]\sum\limits_{i=-N}^{0}U^c(n-1-i,\theta^{i+1}\omega)[G^c(\theta^i\omega,\varphi(i,v,\omega))\\&{}-G^c(\theta^i\omega,\varphi(i,v_0,\omega)) -D
G^c(\theta^i\omega,\varphi(i,v_0,\omega))(\varphi(i,v,\omega)-\varphi(i,v_0,\omega))],
\end{align*}
for $-N \leq n > 0$\,.

Let $\overline N$ be a large  positive number to be chosen later.
For
$|n| < \overline N$\,, we set
\begin{align*}
I_3={}& -\exp[{(-(\eta-\delta))|n|}]\sum\limits_{i=n-1}^{N}U^u(n-1-i,\theta^{i+1}\omega)[G^u(\theta^i\omega,\varphi(i,v,\omega))\\&{}-G^u(\theta^i\omega,\varphi(i,v_0,\omega)) -D
G^u(\theta^i\omega,\varphi(i,v_0,\omega))(\varphi(i,v,\omega)-\varphi(i,v_0,\omega))].
\\
I_3'={}&\exp[{(-(\eta-\delta))|n|}]\sum\limits_{i=-N}^{n-1}U^s(n-1-i,\theta^{i+1}\omega)
 [G^s(\theta^i\omega,\varphi(i,v,\omega))\\&{}-G^s(\theta^i\omega,\varphi(i,v_0,\omega))  -D
G^s(\theta^i\omega,\varphi(i,v_0,\omega))(\varphi(i,v,\omega)-\varphi(i,v_0,\omega))].
\\
I_4={}& -\exp[{(-(\eta-\delta))|n|}]\sum\limits_{i=N}^{\infty}U^u(n-1-i,\theta^{i+1}\omega)[G^u(\theta^i\omega,\varphi(i,v,\omega))\\&{} -G^u(\theta^i\omega,\varphi(i,v_0,\omega)) -D
G^u(\theta^i\omega,\varphi(i,v_0,\omega))(\varphi(i,v,\omega)-\varphi(i,v_0,\omega))].
\\
I_4'={}&\exp[{(-(\eta-\delta))|n|}]\sum\limits_{i=-\infty}^{-N}U^s(n-1-i,\theta^{i+1}\omega)
 [G^s(\theta^i\omega,\varphi(i,v,\omega))\\&{}-G^s(\theta^i\omega,\varphi(i,v_0,\omega))  -D
G^s(\theta^i\omega,\varphi(i,v_0,\omega))(\varphi(i,v,\omega)-\varphi(i,v_0,\omega))].
\end{align*}

For $|n| \geq \overline N$\,, we set
\begin{align*}
I_5={}&-\exp[{(-(\eta-\delta))|n|}]\sum\limits_{i=n-1}^{\infty}U^u(n-1-i,\theta^{i+1}\omega)[G^u(\theta^i\omega,\varphi(i,v,\omega))\\&{}-G^u(\theta^i\omega,\varphi(i,v_0,\omega)) -D
G^u(\theta^i\omega,\varphi(i,v_0,\omega))(\varphi(i,v,\omega)-\varphi(i,v_0,\omega))].
\\
I_5'={}&\exp[{(-(\eta-\delta))|n|}]\sum\limits_{i=-\infty}^{n-1}U^s(n-1-i,\theta^{i+1}\omega)
 [G^s(\theta^i\omega,\varphi(i,v,\omega))\\&{}-G^s(\theta^i\omega,\varphi(i,v_0,\omega))  -D
G^s(\theta^i\omega,\varphi(i,v_0,\omega))(\varphi(i,v,\omega)-\varphi(i,v_0,\omega))].
\end{align*}

It is sufficient to show that for any $\epsilon >0$ there is a
$\sigma
>0$ such that if $ |v-v_0| \leq \sigma $\,, then
$|I|_{{{D_{\eta-\delta}}}} \leq \epsilon |v-v_0|$.
Note that
\begin{align*}
|I|_{{{D_{\eta-\delta}}}}& \leq \sup_{n>N} I_1 + \sup_{N\geq n> 0}
I_2+ \sup_{n< -N} I_1' + \sup_{-N\geq n> 0}
I_2'+ \sup_{ |n|< \overline N} I_3+\sup_{ |n|< \overline N} I_4
+ \sup_{|n|\geq \overline N} I_5\\&\quad  +\sup_{ |n|< \overline N} I_3'+\sup_{  |n|< \overline N} I_4'
+ \sup_{|n |\geq \overline N} I_5'.
\end{align*}
A computation similar to~(\ref{eq(d3.7)}) implies that
\begin{align*}
|I_1|_{D_{\eta-\delta}} &\leq \sum_{i=N}^{n-1} 2K(\theta^{i+1}\omega)\operatorname{Lip} F(\theta^i\omega) \rho(\theta^i\omega)\\&
\exp[{(\gamma-(\eta-\delta))|n-i-1|}]\exp({-\delta |i|})
|\varphi(\cdot, v,\omega)-\varphi(\cdot,v_0,\omega)|_{D_{\eta-2\delta}}  \\
& \leq \frac{2K(\omega)\rho'(\omega) \exp({-\delta
N})}{1-\rho'(\omega) \left[\frac{1}{(\eta-2\delta)-\gamma}+\frac{1}{\beta-(\eta-2\delta)}+\frac{1}{\alpha-(\eta-2\delta)}\right]}
|v - v_0 |.
\end{align*}

 Choose~$N$ so large that
\begin{align*}
\frac{2K(\omega)\rho'(\omega) \exp({-\delta
N})}{1-\rho'(\omega) \left[\frac{1}{(\eta-2\delta)-\gamma}+\frac{1}{\beta-(\eta-2\delta)}+\frac{1}{\alpha-(\eta-2\delta)}\right]}
\leq \frac{1}{8} \epsilon.
\end{align*}

Hence for such~$N$ we have that
\begin{equation*}
\sup_{n> N}I_1 \leq \frac{1}{8}\epsilon |v-v_0|.
\end{equation*}
Fixing such~$N$, for~$I_2$ we have that
\begin{align*}
I_2 &\leq \sum\limits_{i=0}^N\exp[{(-(\eta-\delta))|n|}] K(\theta^{i+1}\omega) \exp[{\gamma|i-1|}]\\
&\quad\Big\{
\int^1_0\big |
D_uG(\theta_i\omega,\tau \varphi(i,v,\omega)+(1-\tau)\varphi(i,v_0,\omega))  -D_uG(\theta_i\omega,\varphi(i,v_0,\omega))\big |\,\Big\}d\tau\\
&\quad
|\varphi(\cdot, v,\omega)-\varphi(\cdot,v_0,\omega)|_{{{C_{\eta-\delta}}}} \, \\
& \leq \frac{K(\omega)|v-v_0|}{1-\rho'(\omega)\left[\frac{1}{\eta-\delta-\gamma}+\frac{1}{\beta-(\eta-\delta)}+\frac{1}{\alpha-(\eta-\delta)}\right]}\\&
\quad \sum\limits_{i=0}^N \exp[{(\gamma-(\eta-\delta))|n-i-1|}]\Big\{ \int^1_0\big |
D_uG(\theta_i\omega,\tau \varphi(i,v,\omega)+(1-\tau)\varphi(i,v_0,\omega))\\
&\quad-D_uG(\theta_i\omega,u(i,v_0,\omega))\big | \,d\tau\Big\}\,ds\,.
\end{align*}
From the continuity of  the integrand in~$(i, \cdot)$,   the last integral is continuous at the point~$v_0$.
Thus, we have that there
is a $\sigma_1
>0$ such that if $|v-v_0| \leq \sigma_1$\,, then
\begin{equation*}
\sup_{N\geq n> 0}I_2 \leq \frac{1}{8}\epsilon |v-v_0|.
\end{equation*}
Therefore, if  $|v-v_0| \leq \sigma_1$\,, then
\begin{equation*}
\sup_{n>N} I_1 + \sup_{N\geq n> 0}
I_2 \leq \frac{1}{4}\epsilon
|v-v_0|.
\end{equation*}
In the same way,   there
is a $\sigma_1'
>0$ such that if  if  $|v-v_0| \leq \sigma_1'$, then
\begin{equation*}
\sup_{n<-N} I_1' + \sup_{-N\leq n < 0}
I_2' \leq \frac{1}{4}\epsilon
|v-v_0|.
\end{equation*}
Similarly, by choosing~$\overline N$ to be sufficiently large,
\begin{equation*}
\sup_{| n|<\overline N} I_4 + \sup_{|n|\geq \overline N}I_5 \leq
\frac{1}{8}\epsilon|v-v_0|,
\end{equation*}
\begin{equation*}
\sup_{| n|<\overline N} I_4' + \sup_{|n|\geq \overline N}I_5' \leq
\frac{1}{8}\epsilon|v-v_0|,
\end{equation*}
and for fixed such~$\overline N$, there exists $\sigma_2>0$ such that
if $|v-v_0| \leq \sigma_2$\,, then
\begin{equation*}
\sup_{  |n|< \overline N}I_3  \leq \frac{1}{8}\epsilon
|v_1-v_2|
\quad\text{and}\quad
\sup_{| n|<\overline N}I_3' \leq \frac{1}{8}\epsilon
|v_1-v_2|.
\end{equation*}
Taking $\sigma = \min \{\sigma_1, \sigma_1', \sigma_2\}$, we have that if
$|v-v_0| \leq \sigma$\,, then
\begin{equation*}
|I|_{{{D_{\eta-\delta}}}} \leq  \epsilon |v-v_0|.
\end{equation*}
Therefore  $|I|_{{{D_{\eta-\delta}}}} = o(|v-v_0|)$ as $v\to
v_0$\,.
The tangency condition $Dh^c(0,\omega)=0$ is from  equation~\eqref{eq(d4.2)}.

 From the definition of~$h^c(v,
\omega)$ and the claim that $ x_0 \in N^c(\omega)$  if and only if
there exists \begin{align*}\varphi(\cdot, x_0, \omega)~\in~D_\eta\end{align*} with $\varphi(0,x_0,\omega)=x_0$ and satisfies~(\ref{eq(d3.2)}) it follows that $ x_0 \in N^c(\omega)$ if and only
if there exists $v \in H^c$ such that $ x_0=v+h^c(v,
\omega)$, therefore,
\begin{equation*}
N^c(\omega) = \{v+ h^c(v, \omega) \mid v \in H^c\}.
\end{equation*}

Next we prove
that for any $x\in H$\,, the function
\begin{equation}\label{eq(d3.8)}
\omega\to\inf_{v\in H^c}\left |x-(v+ h^c(v, \omega))\right |
\end{equation}
is measurable.
Let $H'$~be a countable dense subset of the separable space~$H$.
From  the continuity of~$h^c(\cdot,
\omega)$,
\begin{equation}\label{eq(d3.9)}
\inf_{v\in H^c}|x-(v+ h^c(v, \omega))|
=\inf_{y\in H'}|x-P^cy-h^c(P^cy, \omega)|.
\end{equation}
The measurability of~(\ref{eq(d3.8)}) follows since
$\omega\to h^c(P^cy,\omega)$ is measurable for any $y\in H'$.

Finally, we show that $N^c(\omega)$~is invariant, that is for each
$x_0\in N^c(\omega)$, $\varphi(s, x_0, \omega) \in N^c(\theta_s\omega)$
for all $ s\geq 0$\,.
Since for $s \geq 0$\,,
$\varphi(n,x_s,\theta^s\omega)=\varphi(n+s, x_0,\omega)\in D_\eta$, so $\varphi(s, x_0, \omega) \in
N^c(\theta_s\omega)$.
\end{proof}
Second we prove the eixistenc of center manifolds for continuous time random dynamical systems.
\begin{lemma}
  Let  $U(t,x,\omega)=D\varphi(t,x,\omega)$.  Suppose the random dynamical systems $\varphi(t,x,\omega)$ and $U(t,x,\omega)$ are  differentiable at $t=0$,
  \begin{eqnarray*}
   && f(\omega,x)=\frac d  {dt} \varphi(t,x,\omega)|_{t=0}, \\
  &&  A(\omega)x=\frac d  {dt} U(t,x,\omega)|_{t=0}.
  \end{eqnarray*}
  Then $\varphi(t,x,\omega)$ is the solution of   \begin{eqnarray*}
   && \frac {du}  {dt}=f(\theta_t\omega,u),\\
  &&  u(0,x,\omega)=x.
  \end{eqnarray*}
  $U(t,x,\omega)$ is the solution of
  \begin{eqnarray*}
   && \frac {dv}  {dt}=A(\theta_t\omega)v,\\
  &&  v(0,x,\omega)=x.
  \end{eqnarray*}
\end{lemma}
\begin{proof}
  Since $ \varphi(s+t,x,\omega)= \varphi(s,\cdot,\theta_t\omega)\circ \varphi(t,x,\omega)$, then
   \begin{eqnarray*}
   \frac { \varphi(s+t,x,\omega)-\varphi(t,x,\omega)} s= \frac {\varphi(s,\cdot,\theta_t\omega)\circ \varphi(t,x,\omega)-\varphi(t,x,\omega)} s.
    \end{eqnarray*}
    Let $s\rightarrow 0$ yields
    \begin{eqnarray*}\frac {d\varphi}  {dt}=f(\theta_t\omega,\varphi).  \end{eqnarray*} Similar
     \begin{eqnarray*}
   \frac { D\varphi(s+t,x,\omega)-D\varphi(t,x,\omega)} s= \frac {D\varphi(s,\cdot,\theta_t\omega)\circ D\varphi(t,x,\omega)-D\varphi(t,x,\omega)} s.
   \end{eqnarray*}
   Let $s\rightarrow 0$ yields
   \begin{eqnarray*}
   \frac {dU}  {dt}=A(\theta_t\omega)U.
    \end{eqnarray*}
\end{proof}
Denote $B(\theta_t\omega)x=f(\theta_t\omega,x)-A(\theta_t\omega)x$, then $\varphi$ is the solution of
   \begin{eqnarray}
   && \frac {du}  {dt}=A(\theta_t\omega)u+B(\theta_t\omega)u,\label{4.1}\\
  &&  u(0,x,\omega)=x. \nonumber
  \end{eqnarray}
We  assume the nonlinear term $B(\theta_t\omega)$ satisfies $B(\theta_t\omega)(0) = 0 $, and assume it to be
Lipschitz continuous on $H$, that is,
 \begin{eqnarray*}
|B(\theta_t\omega)u_1 -B(\theta_t\omega)u_2| \leq \operatorname{Lip} B |u_1 - u_2|
 \end{eqnarray*}
with the sufficiently small Lipschitz constant $\operatorname{Lip} B > 0$.
We show the existence of a
center manifold for the random partial differential
equation~\eqref{4.1}.

For each $\eta>0$\,,  we denote the Banach
space
\begin{align*}
C_\eta={}&\left \{ \phi\in C(\mathbb{R}, H) \mid  \sup_{t\in \mathbb{R}}\exp\left[{-\eta |t|}\right] | \phi(t)|
 < \infty\right\}
\end{align*}
with the norm
\[\left|\phi\right|_{C_\eta}=\sup_{t\in \mathbb{R}}\exp\left[{-\eta |t|}\right] |\phi(t)|.
\]
The set~$C_\eta$ is the set of  `slowly varying' functions.
We know that the functions  are controlled by
 $\exp\left[{\eta |t|}\right]$.
Let
\[M^c(\omega)=\left\{u_0\in H \mid u(\cdot, u_0, \omega) \in C_\eta\right\},
\]
where  $u(t, u_0, \omega)$~is  the solution of~(\ref{4.1}) with the
initial data $u (0)= u_0$\,.

  We
prove that $M^c(\omega)$~is invariant and is  the
graph of a Lipschitz function.

Different from the proof process of Duan et al.~\cite{Duan1,Duan2}, we need analysis the behavior of the  solution on center subspace.
\begin{theorem} \label{Thm(3.1)} Suppose $U(t,\omega)$ satisfies the exponential trichotomy.
 If   $\gamma < \eta<\min\{\beta, \alpha\}$ such that the nonlinearity term is sufficiently small,
\begin{eqnarray}K(\omega)  \operatorname{Lip} B  \left(\frac{1}{\eta-\gamma}+\frac{1}{\beta-\eta}+\frac{1}{\alpha-\eta}\right)<
1\,,\label{pr1}\end{eqnarray} then there exists a  center manifold for
 the random  differential  equation~(\ref{4.1}),  which is written as the graph
\[
M^c(\omega) = \{v+h^c(v,\omega)\mid v \in H^c\},
\]
where $h^c(\cdot,  \omega) : H^c\to  H^u\oplus H^s$ is a Lipschitz continuous mapping from the center subspace and
satisfies $h^c(0,\omega)=0$\,. 
\end{theorem}


\begin{proof}
 First we claim that $u_0 \in M^c(\omega)$ if and only if  there exists a slowly varying function $u(\cdot, u_0, \omega)\in
C_\eta$  with
\begin{align}\label{eq(3.2)}\begin{split}
u(t, u_0,\omega)={}&U^c(t,v,\omega)+\int_0^t
U^c(t-s,\omega)P^cB(\theta_r\omega)u(r)\,dr
\\&{} +\int^{t}_{-\infty}U^s(t-r,\omega)P^sB(\theta_r\omega)u(r)\,dr \\&{} -\int_{t}^{+\infty}U^u(t-r,\omega)P^uB(\theta_r\omega)u(r)\,dr\,,
\end{split}
\end{align}
where $v = P^c u_0$.  

To prove this claim,  first we let $u_0 \in M^c(\omega)$.
By
using the variation of constants formula, the solution on each subspace denoted as
\begin{align}&\label{eq(3.3)}
P^cu(t, u_0, \omega) = U^c(t,\omega)
v + \int_0^tU^c(t-r,\omega)
P^cB(\theta_r\omega)u(r)\,dr\,.
\\&
\label{eq(3.4)}
P^u u(t,u_0, \omega) =U^u(t-\tau,\omega)P^uu(\tau, u_0, \omega)+\int_{\tau}^t U^u(t-r,\theta_r\omega)
P^uB(\theta_r\omega)u(r)\,dr\,.
\\&\label{eq(3.44)}
P^s u(t, u_0, \omega) =U^s(t-\tau,\omega)P^s u(\tau, u_0, \omega)+\int_{\tau}^t U^s(t-r,\theta_r\omega)
P^s B(\theta_r\omega)u(r)\,dr\,.
\end{align}
Since the slowly varying function $u \in C_\eta$\,, we have for $t < \tau $ that the magnitude
\begin{align*}
\left|U^u(t-\tau,\omega)P^uu(\tau, u_0, \omega)\right|
&\leq K^u(\omega)\exp[{\alpha(t-\tau)} ]\exp(\eta\tau)|u|_{C_{\eta}} \\
& = K^u(\omega)\exp[{\alpha t} ]
\exp[{-(\alpha-\eta)\tau}] |u|_{C_{\eta}}
\\& \to 0 \quad\text{as }\tau \to +\infty.
\end{align*}
For  $t > \tau $\,,
\begin{align*}
\left|U^s(t-\tau,\omega)P^uu(\tau, u_0, \omega)\right|
&\leq K^s(\omega)\exp[{-\beta(t-\tau)} ]\exp(\eta\tau)|u|_{C_{\eta}} \\
& =K^s(\omega)\exp[{-\beta t} ]
\exp[{(\beta+\eta)\tau}] |u|_{C_{\eta}}
\\&\to 0 \quad\text{as }\tau \to -\infty.
\end{align*}
Then, taking the two separate limits $\tau \to \pm  \infty$   in~(\ref{eq(3.4)}) and~\eqref{eq(3.44)} respectively,
\begin{align}&\label{eq(3.5)}
P^uu(t, u_0, \omega)=\int^t_\infty
 U^u(t-r,\theta_r\omega)
P^uB(\theta_r\omega)u(r)\,dr,
\\&\label{eq(3.55)}
P^su(t, u_0, \omega)=\int^t_{-\infty}
U^s(t-r,\theta_r\omega)
P^s B(\theta_r\omega)u(r)\,dr\,.
\end{align}
Combining~(\ref{eq(3.3)}),  (\ref{eq(3.5)}) and~(\ref{eq(3.55)}), we have~(\ref{eq(3.2)}).
The converse  follows from a direct computation.

Next we prove  that for any given $v \in H^c$, the centre subspace, the integral
equation~(\ref{eq(3.2)}) has a unique solution in the slowly varying functions space~$C_{\eta}$.
Let
\begin{align}\label{eq(3.22)}\begin{split}
J^c(u,v):={}&U^u(t,\omega)v+\int_0^t
U^c(t-r,\omega)
P^cB(\theta_r\omega)u(r)\,dr\\&{}
 +\int^{t}_{-\infty}U^s(t-r,\omega)P^sB(\theta_r\omega)u(r)\,dr\\&{} -\int_{t}^{+\infty}U^u(t-r,\omega)P^uB(\theta_r\omega)u(r)\,dr\,.
\end{split}
\end{align}
$ J^c$~is
well-defined from $C_{\eta}\times H^c$ to the slowly varying functions space~$C_{\eta}$.
For each pair of slowly varying functions
$u, \bar u \in C_{\eta}$\,, we have  that for $\gamma < {\eta}<\min\{\beta, \alpha\}$, $K(\omega)=\max\{K^s(\omega),K^c(\omega),K^u(\omega)\}$,
\begin{align}\label{eq(3.6)}
\begin{split}
&| J^c(u,v ) -  J^c (\bar u, v ) |_{C_{\eta}} \\
&\leq\sup_{t\in \mathbb{R}} \Bigg\{\exp\left[
-{\eta} |t| \right] \bigg |\int_0^t
U^c(t-r,\omega)
P^c(B(\theta_r\omega)u-B(\theta_r\omega)\bar u)\,dr
\\& \quad
+\int^{t}_{-\infty}U^s(t-r,\omega)(P^sB(\theta_r\omega)u-P^sB(\theta_r\omega)\bar u)\,dr\\& \quad +\int_{t}^{+\infty}U^u(t-r,\omega)(P^uB(\theta_r\omega)u-P^uB(\theta_r\omega)\bar u)\,dr\bigg|
\Bigg\}\\
&\leq\sup_{t\in \mathbb{R}} \Bigg\{K(\omega) \operatorname{Lip} B |u-\bar
u|_{C_{\eta}}\bigg|\int_0^{t} \exp[{(\gamma-{\eta})|t-r|}]\,dr
\\&  \quad
+\int^t_{-\infty}\exp[{({\eta}-\alpha)(t-r)}]\,dr
+\int_t^{+\infty}\exp[{(\beta-{\eta})(t-r)}]\,dr\bigg |
\Bigg\}\\
&\leq K(\omega) \operatorname{Lip} B \left(\frac{1}{{\eta}-\gamma}+\frac{1}{\alpha-{\eta}}+\frac{1}{\beta-{\eta}}\right) |u-\bar
u|_{C_{\eta}}.
\end{split}
\end{align}

From equation~\eqref{eq(3.6)}, $ J^c$~is Lipschitz continuous in~$v$.
By the theorem's precondition~\eqref{pr1},
$ J^c$~is a uniform contraction with respect to the parameter~$v$.
By the uniform contraction mapping principle,
for each $v \in H^c$,  the mapping $J^c (\cdot , v )$ has a
unique fixed point $u(\cdot,v, \omega) \in C_{\eta}$\,.
Combining equation~\eqref{eq(3.22)} and equation~\eqref{eq(3.6)},
\begin{align}\label{eq(3.7)}
&|u(\cdot ,  v,   \omega) - u (\cdot,  \bar v, \omega
)|_{C_{\eta}}\nonumber \\&{}\quad \leq \frac{K(\omega)}{1-K(\omega)\operatorname{Lip} B \left(\frac{1}{{\eta}-\gamma}+\frac{1}{\beta-{\eta}}+\frac{1}{\alpha-{\eta}}\right)}|v - \bar v |,
\end{align}
for each
 fixed point~$u(\cdot,v, \omega)$.
 Then for each time $t$\,,
$u(t,   \cdot, \omega )$~is Lipschitz from the center subspace~$H^c$ to slowly varying  functions~$C_{\eta}$.
 $u(\cdot,  v, \omega  )\in C_{\eta}$ is a unique solution
of the integral equation~(\ref{eq(3.2)}).
Since $u(\cdot , v,   \omega)$ can be an $\omega$-wise limit of the
iteration of contraction mapping~$J^c$ starting at~$0$ and $J^c$
maps a $\mathcal{F}$-measurable function to a $\mathcal{F}$-measurable function,
$u(\cdot, v,   \omega)$~is $\mathcal{F}$-measurable.
Combining $u(\cdot, v, \omega)$~is  continuous with respect to~$H$, we have $u(\cdot , v,   \omega)$ is measurable with respect to~$(\cdot, v, \omega)$.

Let $h^c (v,\omega) := P^s u (0, v, \omega)\oplus P^u u (0, v, \omega)$.
Then
\begin{align*}
h^c(v, \omega)&=\int^{0}_{-\infty}U^s(-r,\omega)P^sB(\theta_r\omega)u(r,v,\omega)\,dr\\&{}\qquad-\int_{0}^{+\infty}U^u(-r,\omega)P^uB(\theta_r\omega)u(r,v,\omega)\,dr\,.
\end{align*}
We see that~$h^c$ is $\mathcal{F}$-measurable and  $h^c(0,\omega)=0$\,. We prove $Dh^c(0,\omega)=0$\,.
Since
\[
K(\omega) \Lip B\left(\frac{1}{\eta-\gamma}+\frac{1}{\beta-\eta}+\frac{1}{\alpha-\eta}\right)< 1
\]
there exists a small number $\delta>0$ such that $\gamma <
\eta-\eta'<\min\{\beta,\alpha\}$ and for all  $ 0 \leq \eta' \leq 2\delta$\,,
\[
K(\omega) \Lip B \left[\frac{1}{(\eta-\eta')-\gamma}+\frac{1}{\beta-(\eta-\eta')}+\frac{1}{\alpha-(\eta-\eta')}\right]<
1\,.
\]
Thus, $J^s(\cdot,v)$ is a uniform contraction in $C_{\eta-\eta'}(\omega)
\subset  C_\eta(\omega)$ for any $0\leq \eta'\leq 2\delta$\,.
Therefore,
$u(\cdot, v,\omega)\in C_{\eta-\eta'}(\omega)$\,.
For $v_0\in H^c$, we
define two operators: let
\[
\mathcal{S}v_0=U^u(t,\omega)v_0 \,,
\]
and for $u'\in {C_{\eta-\delta}}$ let
\begin{align*}
\mathcal{T} u'&=\int_0^tU^c(t-r,\omega)
P^cDB(\theta_r\omega),u(r,v_0,\omega))u'\, dr \\
&\quad {}+\int_{-\infty}^{t} U^s(t-r,\omega)
P^sDB(\theta_r\omega,u(r,v_0,\omega))u' \,dr\\
&\quad {}-\int^{\infty}_{t} U^u(t-r,\omega)
P^uDB(\theta_r\omega,u(r,v_0,\omega))u'\,dr.
\end{align*}
From the assumption,
$\mathcal{S}$~is a bounded linear operator from center subspace~$H^c$ to slowly varying functions space~$C_{\eta-\delta}$.
Using the same arguments   that
$J^c$~is a contraction, we have that $\mathcal{T}$~is a bounded linear
operator from~${C_{\eta-\delta}}$ to itself and
\[
\|\mathcal{T}\| \leq K(\omega) \Lip B \left(\frac{1}{\eta-\delta-\gamma}+\frac{1}{\beta-(\eta-\delta)}+\frac{1}{\alpha-(\eta-\delta)}\right)<
1\,,
\]
which implies that the operator $\operatorname{id}-\mathcal{T}$ is invertible in~${C_{\eta-\delta}}$.
For $v, v_0\in H^c$, we set
\begin{align*}
I={}&\int_0^t U^c(t-r,\omega)
P^c\Big[B(\theta_r\omega,u(r,v,\omega))-
B(\theta_r\omega,u(r,v_0,\omega))\\
&\qquad{}-DB(\theta_r\omega,u(r,v_0,\omega))
(u(r,v,\omega)-u(r,v_0,\omega))\Big]\,dr\\
&{} +\int^{t}_{-\infty}U^s(t-r,\omega)P^s
\Big[B(\theta_r\omega,u(r,v,\omega))-
B(\theta_r\omega,u(r,v_0,\omega))\\
&\qquad{}-DB(\theta_r\omega,u(r,v_0,\omega))
(u(r,v,\omega)-u(r,v_0,\omega))\Big]\,dr
\\
&{} -\int_{t}^{\infty}U^u(t-r,\omega)P^u
\Big[B(\theta_r\omega,u(r,v,\omega))-
B(\theta_r\omega,u(r,v_0,\omega))\\
&\qquad{}-DB(\theta_r\omega,u(r,v_0,\omega))
(u(r,v,\omega)-u(r,v_0,\omega))\Big]\,dr\,.
\end{align*}

We obtain
\begin{align}\label{eq(4.1)}
\begin{split}
&u(\cdot,v, \omega)-u(\cdot,v_0,
\omega)-\mathcal{T} (u(\cdot,v,
\omega)-u(\cdot,v_0,
\omega))
=\mathcal{S}(v-v_0) + I,\\
\end{split}
\end{align}
which yields
\begin{align*}
u(\cdot,v, \omega)-u(\cdot,v_0,
\omega)=(\operatorname{id}-\mathcal{T})^{-1}\mathcal{S}(v-v_0)+(\operatorname{id}-\mathcal{T})^{-1}I.
\end{align*}
Using the same approach as the discrete case,  $|I|_{{{C_{\eta-\delta}}}}=o(|v-v_0|)$ as $v\to
v_0$, then  $u(\cdot,v, \omega) $ is differentiable in~$v$ and its
derivative satisfies $D u(t, v,\omega)\in L(H^c,
{{C_{\eta-\delta}}})$,  where $L(H^c, {{C_{\eta-\delta}}})$ is the
usual space of bounded linear operators and
\begin{align}\label{eq(4.2)}
\begin{split}
D u(t, v,\omega)& =U^u(t,\omega)v   +\int_0^t U^c(t-r,\omega)
P^cDB(\theta_r\omega,u(r,v,\omega))D u(r,v,\omega)\,dr\\& \quad+\int^{t}_{-\infty} U^s(t-r,\omega)P^sDB(\theta_r\omega,u(r,v,\omega)) Du(r,v,\omega) \,dr
\\& \quad+\int^{\infty}_{t} U^u(t-r,\omega)P^uDB(\theta_r\omega,u(r,v,\omega)) D u(r,v,\omega) \,dr\,.
\end{split}
\end{align}

The tangency condition $Dh^c(0,\omega)=0$ is from  equation~\eqref{eq(4.2)}.
 From the definition of~$h^c(v,
\omega)$ and the claim that $ u_0 \in M^c(\omega)$  if and only if
there exists \begin{align*}u(\cdot, u_0, \omega)~\in~C_{\eta}\end{align*}with $u(0)=u_0$ and satisfies~(\ref{eq(3.2)}) it follows that $ u_0 \in M^c(\omega)$ if and only
if there exists $v \in H^c$ such that $ u_0=v+h^c(v,
\omega)$, therefore,
\begin{equation*}
M^c(\omega) = \{v+ h^c(v, \omega) \mid v \in H^c\}.
\end{equation*}

Next we prove
that for any $x\in H$\,, the function
\begin{equation}\label{eq(3.8)}
\omega\to\inf_{y\in H^c}\left |x-(y+ h^c(y, \omega))\right |
\end{equation}
is measurable.
Let $H'$~be a countable dense subset of the separable space~$H$.
From  the continuity of~$h^c(\cdot,
\omega)$,
\begin{equation}\label{eq(3.9)}
\inf_{y\in H^c}|x-(y+ h^c(y, \omega))|
=\inf_{y\in H'}|x-P^cy-h^c(P^cy, \omega)|.
\end{equation}
The measurability of~(\ref{eq(3.8)}) follows since
$\omega\to h^c(P^cy,\omega)$ is measurable for any $y\in H'$.

Finally, we show that $M^c(\omega)$~is invariant, that is for each
$u_0\in M^c(\omega)$, $u(s, u_0, \omega) \in M^c(\theta_s\omega)$
for all $ s\geq 0$\,.
Since for $s \geq 0$\,,
$u(t+s, u_0,\omega)$ is a solution of
\[
\frac{du}{dt}=A(\theta_t(\theta_s\omega))u+B(\theta_t(\theta_s\omega))u,\quad
u(0)=u(s, u_0, \omega).
\]
Thus $u(t,u(s,u_0,\omega), \theta_s\omega)=u(t+s, u_0,\omega)$ and  $u(t,u(s,u_0,\omega),
\theta_s\omega) \in C_{\eta}$\,.
So we conclude  $u(s, u_0, \omega) \in
M^c(\theta_s\omega)$.

\end{proof}

\begin{corollary} \label{cor(3.2)} Suppose the linear random dynamical systems $U(t,\omega)$ satisfies the  multiplicative ergodic theorem (\textsc{met}) and  $\alpha, \beta, \gamma \in\mathbb{R}$ is not contained in the
Lyapunov spectrums such that $\cdots<\lambda_i<-\beta<-\gamma<\lambda_j<\gamma<\alpha<\cdots<\lambda_2<\lambda_1$.
\begin{eqnarray}K(\omega)  \operatorname{Lip} B  \left(\frac{1}{\eta-\gamma}+\frac{1}{\beta-\eta}+\frac{1}{\alpha-\eta}\right)<
1\,,\label{pr1}\end{eqnarray} then there exists a  center manifold for
 the random  random dynamical system $\varphi(t,x,\omega)$,  which is written as the graph
\[
M^c(\omega) = \{v+h^c(v,\omega)\mid v \in H^c\},
\]
where $h^c(\cdot,  \omega) : H^c\to  H^u\oplus H^s$ is a Lipschitz continuous mapping from the center subspace and
satisfies $h^c(0,\omega)=0$\,. 
exponential trichotomy condition  and $\text{{\rm
Lip}}_u G$ denotes the Lipschitz constant of $G(\cdot, u)$ with
respect to $u$.
\end{corollary}
\section{Applications} \label{Sec4}
\begin{example}\label{exam1}
In this example, we show the  above center manifold theory by illustrating a stochastic evolution equation
\begin{equation}\label{eq(2.1)}
    \frac{d u}{dt}=A u+F( u)+u\circ \dot{W}(t),
\end{equation}
where $u\in H$ is a Hilbert space typically defined on some spatial domain,  ${W}(t)$~is the standard $\mathbb{R}$-valued Wiener process on a probability space~$(\Omega,\mathcal{F}, \mathbb{P} )$,  which is only	dependent on  time. Suppose the spectrum of~$A$ satisfies\[
\mu_{1}>\cdots>\mu_{j}>0>\mu_{j+1}>\mu_{j+2}>\cdots\text{
\ \ (with \ }\mu_{j}\rightarrow-\infty\text{ \ as \ }j\rightarrow
\infty\text{)},%
\] 
and $A$~generates a strong continuous semigroup $S_A(t)$, being $S_A(t)$ compact for all $t\geq 0$.
 We assume the nonlinear term~$F$ satisfies $F(0)=0$\,,  and assume it to be Lipschitz
continuous on~$H$, that is,
\[
|F(u_1)-F(u_2)|\le \Lip  F\,|u_1-u_2 |
\]
with the sufficiently small  Lipschitz constant $\Lip  F>0$\,.

  The above example  has been show that there exist a stochastic center manifold by Chen et al.~\cite{Chen}.
 Now we only verify the equation \eqref{eq(2.1)} satisfies the \textsc{met} in Section \ref{pre}, then  there exists a stochastic center manifold.

Let $C_0(\mathbb{R},\mathbb{R})$ be continuous functions on~$\mathbb{R}$, the associated distribution~$\mathbb{P}$ is a Wiener measure defined on the Borel-$\sigma$-algebra $\mathcal{B}(C_0(\mathbb{R},\mathbb{R}))$.
Define $\{\theta_t\}_{t\in\mathbb{R}}$ to be the metric dynamical system generated by the Wiener process~$W(t)$.

We transform  the stochastic evolution equation~\eqref{eq(2.1)}  into the following partial differential equation with random coefficients
\begin{equation}\label{eq(2.5)}
\frac{du}{dt}=Au+z(\theta_t\omega)u+G(\theta_t\omega,u),\quad
u(0)=u_0\in H,
\end{equation}
where $G(\omega,u)=\exp[{-z(\omega)}]F(\exp[{z(\omega)}]u)$, $z(\theta_t\omega)$ is the solution of Ornstein--Uhlenbeck equation,
\begin{equation}\label{eq(2.4)}
dz+ z\,dt=dW.
\end{equation}
And
\[
|G(\omega,u_1)-G(\omega,u_2)|\le \Lip_uG\,  |u_1-u_2 |,
\]
where $\Lip_u G$ denotes the Lipschitz constant of~$G(\cdot, u)$ with
respect to~$u$.
For  any $\omega\in\Omega$ the function~$G$ has the
same global Lipschitz constant  as~$F$ by the construction of~$G$.
The linearization equation is
\begin{equation}\label{eq(2.5)}
\frac{du}{dt}=Au+z(\theta_t\omega)u,\quad
u(0)=u_0\in H,
\end{equation}
$U(t,\omega)$ is compact since $S_A(t)$ is compact. We prove that the assumption \eqref{inte} is satisfied. For $t \in  [-1,1],$
\begin{align*}
\|U(t,\omega)\| & \leq\|S_A(t)\| |
\exp\int_{0}^{t}z(\theta_{r}\omega)dr|,
 \\
\|U(1-t,\theta_t\omega)\|
&  \leq\|S_A(1-t)\|  |\exp  \int_{0}^{1-t}z(\theta_{t+r}\omega)dr|,
\end{align*}
and%
\begin{align*}
\log^{+}\|U(t,\omega)\| &  \leq\log^{+}\|S_A(t)\| +\int_{-1}^{1}|z(\theta_{r}\omega)|dr,\\
\log^{+} \|U(1-t,\theta_t\omega)\| & \leq \log^{+} \|S_A(1-t)\|+
\int_{-2}^{1}|z(\theta_{r}\omega)|dr,
\end{align*}
Therefore,%
\[
 \mathbb{E}\sup_{0\le t\le 1}\log^+\|U^{\pm}(t,\omega)\|+
    \mathbb{E}\sup_{0\le t\le
    1}\log^+\|U^{\pm}(1-t,\theta_t\omega)\|<\infty.
\]
\end{example}
\begin{example}
  Here we consider the Burgers' equation with a random force which is a space-time white noise
  \begin{eqnarray}\label{eqaa1}
    \frac{\partial u}{\partial t}=\triangle u-u\nabla u+u+\sigma \frac{\partial  \phi(x,t)}{\partial t \partial x},\\
     u(0,t)=u(\pi,t)=0,\nonumber \\
      u(x,0)=u_0.\nonumber
  \end{eqnarray}
  Let $H=L^2(0,\pi)$, and consider a cylindrical Wiener process by setting
  \begin{eqnarray*}
     W(t)=\frac {\partial \phi}{\partial x}=\sum_{k=1}^{\infty}\sigma_ke_kW_k(t), 
  \end{eqnarray*}
  where $e_k$ is an orthonormal basis of  $H$ and $W_k(t)$ is a sequence of mutually independent real Brownian motions in a fixed probability space $(\Omega,\mathcal{F},\mathbb{P})$ adapted to a filtration $\{\mathcal{F}_t\}_{t\geq 0}$, $\sum_{k=1}^{\infty}\sigma_k<\infty$. Then we rewrite the equation \eqref{eqaa1} to
    \begin{eqnarray}\label{eqaa2}
    d u =\triangle u-u\nabla u+u+\sigma  dW,\\
     u(x,0)=u_0.\nonumber
  \end{eqnarray}
 Mohammed et al.~\cite{Moh} have shown that the \textsc{spde} \eqref{eqaa2} generates a $C^1$ perfect cocycle $\varphi(t,x,\omega)$. Linearize the dynamics of the  the \textsc{spde} \eqref{eqaa2},
  \begin{eqnarray}\label{eqaa3}
    d u =\triangle u+u+\sigma  dW,\\
     u(x,0)=u_0.\nonumber
  \end{eqnarray} From the example \ref{exam1}, the  assumption \eqref{inte} is satisfied.    The Lyapunov exponents is $\lambda_k=-(k^2-1)$, $k=1,2,\cdots$. By a cut-off technique, we claim there exists a local center manifold $M^c(\omega)$.
  Now we directly use  computer algebra to compute the stochastic center manifold~\cite{Rob1}.
Define the Ornstein--Uhlenbeck processes $\mathcal{H}_k\phi=\exp[-(k^2-1)t]\star \dot W(t)=\int_{-\infty}^t\exp[-(k^2-1)(t-\tau)]\phi(\tau)d\tau$,
 then
\begin{align}
    u=&
    a\sin x -\frac 16a^2\sin 2x +\frac 1{32}a^3\sin3x
    +\sigma\sum_{k=2}^\infty \mathcal{H}_k \phi_k\sin kx
    \nonumber\\
    &{}+a\sigma\Big[ -\frac 16\mathcal{H}_2\phi_2\sin x
    +(\frac 13\mathcal{H}_2\phi_1 +\mathcal{H}_2\mathcal{H}_3\phi_3)\sin2x
\nonumber \\&{}+\sum_{k=3}^\infty \frac k2\mathcal{H}_k( \mathcal{H}_{k+1}\phi_{k+1}
   -\mathcal{H}_{k-1}\phi_{k-1}) \sin kx \Big]+\mathcal{O}{(a^4+\sigma^2)}\,.
    \label{eq:lincm}
\end{align}
The corresponding model for the evolution,
\begin{equation}
    \dot a =
    -\frac 1{12}a^3 +\sigma\phi_1
    +\frac 16a\sigma\phi_2
    +a^2\sigma(\frac 1{18}\phi_1+\frac 1{96}\phi_3)
    +\mathcal{O}{(a^5+\sigma^2)}\,,
    \label{eq:linmod}
\end{equation}
has no fast-time convolutions.
\end{example}
\paragraph{Acknowledgement} This work was supported by the Australian Research Council   grants DP0774311 and DP0988738, and by the NSF grant 1025422.

\end{document}